\documentclass[letterpaper]{article}

\usepackage[utf8x]{inputenc}
\usepackage[T1]{fontenc}
\usepackage[margin=1in]{geometry}

\usepackage{amsmath,amssymb,amsthm,bbm,bm}
\usepackage{mathrsfs}
\usepackage[linesnumbered,ruled,noend]{algorithm2e}

\newtheorem{Remark}{Remark}

\newtheorem{Lemma}{Lemma}

\newtheorem{Theorem}{Theorem}

\usepackage{graphicx}
\renewcommand{\Re}{\mathbb{R}}

\newcommand{\Expectation}{\mathbb{E}}
\newcommand{\trace}{\mathtt{tr}}

\newcommand{\ScriptA}{\mathcal{A}}
\newcommand{\ScriptB}{\mathcal{B}}
\newcommand{\ScriptD}{\mathcal{D}}

\newcommand{\erf}{\mathtt{erf}}
\newcommand{\minimize}{\textrm{minimize}}

\usepackage[subrefformat=parens,labelformat=parens]{subfig}

\title{\LARGE \bf Input Hard Constrained Optimal Covariance Steering}

\author{
	\thanks{This work has been supported by NSF award CPS-1544814 and ARL DCIST CRA  W911NF-
17-2-0181. The first author was also partially
supported by the Funai Foundation for Information Technology.}%
	Kazuhide Okamoto\thanks{K. Okamoto is with the School of Aerospace Engineering, Georgia Institute of Technology, Atlanta, GA 30332-0150, USA Email: {\small kazuhide@gatech.edu}}%
	\qquad 
	Panagiotis Tsiotras
	\thanks{P. Tsiotras is with the School of Aerospace Engineering, and with the Institute for Robotics and Intelligent Machines, Georgia Institute of Technology, Atlanta, GA 30332-0150, USA Email: {\small tsiotras@gatech.edu}}
}

\begin{document}
	\maketitle
	\thispagestyle{empty}
	\pagestyle{empty}
	\begin{abstract}
	We address the optimal covariance steering (OCS) problem for stochastic discrete linear systems with additive  Gaussian noise under state chance constraints and input hard constraints. 
	Because the system state can be unbounded due to the unbounded noise, the state constraints are formulated as probabilistic (chance) constraints, i.e., the maximum probability of constraint violation is constrained.
	In contrast, because it is hard to interpret the appropriate control action when the control command violates the constraints, probabilistically formulating the control constraints are difficult, and deterministic hard constraints are preferable.
	In this work we introduce an OCS approach subject to simultaneous state chance constraints and input hard constraints and validate the approach using numerical simulations.
	\end{abstract}

\section{Introduction}\label{sec:Introduction}

The optimal covariance steering (OCS) problem is a stochastic optimal control problem such that a controller steers the covariance of a stochastic system to a target terminal value, while minimizing a state and control dependent cost in expectation.
Since the late `80s~\cite{hotz1985covariance, hotz1987covariance}, the infinite horizon case for linear time invariant systems has been thoroughly researched.
On the other hand, until recently, the finite horizon case has not been investigated~\cite{chen2016optimalI,bakolas2016optimalCDC,bakolas2018finite,beghi1996relative,goldshtein2017finite}.
Our previous work~\cite{okamoto2018Optimal} introduced state chance constraints into the OCS problem.
While it is possible to separately design mean steering (feedforward) and covariance steering (feedback) controllers, when state chance constraints exist, the state mean and state covariance constraints are coupled, and one needs to simultaneously design the mean and covariance steering controllers. 
In addition, in~\cite{okamoto2018Optimalb} we proposed an OCS controller that is computationally efficient and can deal with non-convex state chance constraints; the algorithm was then applied to vehicle path planning problems.

Note that the above mentioned OCS controllers cannot deal with input hard constraints, because they are affine functions of the state or the disturbance, which are both unbounded, and the control commands become unbounded as well.
Thus, they cannot satisfy input hard constraint specifications. 
Such a situation occurs in many real-world scenarios. 
For example, in aircraft~\cite{okamoto2018Optimal} or spacecraft~\cite{Ridderhof2019MinimumFuel} control problems the control command values have physical restrictions such as minimum/maximum thrust.

The contribution of this article is the introduction of an OCS controller that can accommodate input \emph{hard} constraints.
To the best of our knowledge, input hard constraints have not been discussed in the framework of OCS, while input \emph{soft} constraints have been investigated by several prior works. 
These include the incorporation of a maximum value of the expectation of quadratic functions of the control inputs~\cite{bakolas2016optimalCDC,bakolas2018finite}, and the maximum probability of control constraint violation~\cite{Ridderhof2019MinimumFuel}. 
Our formulation of the input constraint is different than these approaches, in the sense that we directly impose hard constraints on the input.
Inspired by the approach in~\cite{paulson2017stochastic,hokayem2012stochastic}, we propose to use a saturation function into our OCS controller~\cite{okamoto2018Optimalb} to impose input hard constraints. 

The remainder of this paper is organized as follows.
Section~\ref{sec:ProblemStatement} formulates the problem and introduces the input constrained OCS problem setup.
In Section~\ref{sec:ProposedApproach}, we introduce the newly developed input hard constrained OCS controller. 
Section~\ref{sec:Numerical Simulation} validates the effectiveness of the proposed approach using numerical simulations. 
Finally, Section~\ref{sec:Summary} summarizes the results and discusses some possible future research directions. 

We use $P \succ (\succeq)~0$ to denote that the matrix $P$ is symmetric positive (semi)-definite.
Also, we use $P > (\geq)~0$ to denote element-wise inequalities of the matrix $P$.
The trace of $P$ is denoted with $\trace(P)$, and $\mathtt{blkdiag}(P_0,\ldots,P_N)$ denotes the block-diagonal matrix with block-diagonal matrices $P_0,\ldots, P_N$.
$\|v\|$ is the 2-norm of the vector $v$, and $\|P\|$ is the induced 2-norm of the matrix~$P$.
$x \sim \mathcal{N}(\mu, \Sigma)$ denotes a random variable $x$ sampled from a Gaussian distribution with mean $\mu$ and (co)variance $\Sigma$.
$\Expectation[x]$ denotes the expected value, or the mean, of the random variable $x$. 
Finally, $\erf(\cdot)$ is the error function.

\section{Problem Statement\label{sec:ProblemStatement}}
In this section we formulate the input-hard constrained OCS problem.

\subsection{Problem Formulation}
We consider the following discrete-time stochastic linear time-varying system with additive noise
\begin{equation}  \label{eq:SystemDynamics}
	x_{k+1} = A_kx_k + B_ku_k + w_k, 
\end{equation}
where $k$ denotes the time index, $x\in\Re^{n_x}$ is the state, $u\in\Re^{n_u}$ is the control input, and $w\in\Re^{n_x}$ denotes a random vector distributed according to a zero-mean white Gaussian noise with the properties,
\begin{equation}
	\mathbb{E}\left[w_k\right] = 0, \
	\mathbb{E}\left[w_{k_1}{w_{k_2}^{\top}}\right] = 
	\begin{cases} D_{k_1}D_{k_2}^\top, &\mbox{if\ } k_1 = k_2,\\
	0, &\mbox{otherwise.}
	\end{cases}\label{eq:GaussianNoise}
\end{equation}
In addition, $A_k$, $B_k$, and $D_k$ are known system matrices having appropriate dimensions. 
Furthermore, we assume that
\begin{equation}
	\mathbb{E}\left[x_{k_1}w_{k_2}^{\top}\right]=0,\qquad0 \leq k_1 \leq k_2 \leq N.
\end{equation}
We also assume that the state and control inputs are subject to the constraints
\begin{align}\label{eq:state_and_input_const}
	x_k \in \mathcal{X},\quad u_k \in \mathcal{U},
\end{align}
for all $k \geq 0$, where $\mathcal{X} \subseteq \Re^{n_x}$ and $\mathcal{U}\subseteq \Re^{n_u}$ are convex sets containing the origin.
Throughout this work, we assume that the sets $\mathcal{X}$ and $\mathcal{U}$ are represented as an intersection of $N_s$ and $N_c$ linear inequality constraints, respectively, as follows
\begin{align}
	\mathcal{X} &= \bigcap_{j = 0}^{N_s-1} \left\{x: \alpha_{x,j}^\top x \leq \beta_{x,j}\right\},\label{eq:Xdefinition}\\
	\mathcal{U} &= \bigcap_{s = 0}^{N_c-1} \left\{u: \alpha_{u,s}^\top u \leq \beta_{u,s}\right\},\label{eq:Udefinition}
\end{align}
where $\alpha_{x,j}\in\Re^{n_x}$ and $\alpha_{u,s}\in\Re^{n_u}$ are constant vectors, and $\beta_{x,j} \in \Re$ and $\beta_{u,s} \in \Re$ are constant scalars.
Notice that, since the system noise in~(\ref{eq:SystemDynamics}) is possibly unbounded, the state is unbounded as well. 
Thus, we formulate the state constraints $x_k \in \mathcal{X}$ probabilistically, as chance constraints,
\begin{align}\label{eq:origCC}
	\Pr(x_k \notin \mathcal{X}) \leq \epsilon,
\end{align}
where $\epsilon \in (0, 1)$.
We keep the second set inclusion in~(\ref{eq:state_and_input_const}) since for the control inputs, hard constraints are preferable.
Using Boole's inequality, (\ref{eq:Xdefinition}) and (\ref{eq:origCC}) can be satisfied if
\begin{align}\label{eq:stateChance}
	\Pr\left(\alpha_{x,j}^\top x_k > \beta_{x,j} \right) &\leq p_{x,j},
\end{align}
for $j=0,\ldots,N_s-1$, where $p_{x,j}$ are pre-specified such that 
\begin{align}\label{eq:p_xj}
	\sum_{j=0}^{N_s-1} p_{x,j}\leq \epsilon.
\end{align}

The initial state $x_0 \in \Re^{n_x}$ is a random vector that is drawn from a normal distribution according to 
\begin{equation}\label{eq:x0}
	x_0\sim\mathcal{N}(\mu_0,\Sigma_0),
\end{equation} 
where $\mu_0 \in \Re^{n_x}$ and $\Sigma_0 \in \Re^{n_x \times n_x}$. We assume that $\Sigma_0 \succeq 0$.

Our objective is to design a control sequence $\{u_0,\ldots,u_{N-1}\}$ that steers the system state $x_k$ to a target Gaussian distribution at time step $N$, that is,
\begin{align}\label{eq:xf}
	x_N \sim \mathcal{N}(\mu_f, \Sigma_f),
\end{align}
where $\mu_f \in \Re^{n_x}$, $\Sigma_f \in \Re^{n_x \times n_x}$.
We assume that $\Sigma_f \succ 0$ and wish to minimize the state and control expectation-dependent quadratic cost
\begin{align}\label{eq:originalCostFunc}
	J(u_0,\ldots,u_{N-1}) = \mathbb{E}\left[\sum_{k=0}^{N-1}x_k^\top Q x_k + u_k^\top R u_k\right],
\end{align}
where $Q \succeq 0$ and $R \succ 0$.

In summary, we wish to solve the following finite-horizon optimal control problem
\begin{subequations}\label{prob:FiniteHorizonOptimalCovarianceSteeringProblem}
	\begin{align}
	&\minimize \nonumber\\ 
	&J(u_0,\ldots,u_{N-1}) = \mathbb{E}\left[\sum_{k=0}^{N-1}x_k^\top Q x_k + u_k^\top R u_k\right], \\
	&\textrm{subject to}\nonumber \\
	& \hspace{8pt} x_{k+1} = A_kx_k + B_ku_k + w_k, \\
	& \hspace{8pt} \Pr\left(\alpha_{x,j}^\top x_k > \beta_{x,j} \right) \leq p_{x,j},\; \forall j \in [0,N_s-1], \\
	& \hspace{8pt} \alpha_{u,s}^\top u_k \leq \beta_{u,s},\; \forall s \in [0,N_c-1],\label{eq:OCSInputConstraint}\\
	& \hspace{8pt} x_N = x_f \sim \mathcal{N}(\mu_f,\Sigma_f).
	\end{align}
\end{subequations}

\begin{Remark}
	System (\ref{eq:SystemDynamics}) is assumed to be controllable in the absence of constraints and disturbances, that is, $x_f$ is reachable from $x_0$ for any $x_f \in \Re^{n_x}$, provided that $w_k = 0$ for $k = 0,\ldots,N-1$. This assumption implies that given any $x_f \in \Re^{n_x}$ and $x_0 \in \Re^{n_x}$, there exists a sequence of control inputs $\{u_0,\ldots,u_{N-1}\}$ that steers $x_0$ to $x_f$ in the absence of disturbances or any constraints.
\end{Remark}

\section{Proposed Approach\label{sec:ProposedApproach}}

This section introduces the proposed approach to solve Problem~(\ref{prob:FiniteHorizonOptimalCovarianceSteeringProblem}).
Instead of~(\ref{eq:SystemDynamics}), we use the following equivalent form of the system dynamics
\begin{equation}\label{eq:X=Ax0+BU+DW}
	X = \ScriptA x_{0} + \ScriptB U+\ScriptD W,
\end{equation}
where $X\in\Re^{(N+1) n_x}$, $U\in\Re^{Nn_u}$, and $W\in\Re^{Nn_x}$ represents the concatenated state, input, and disturbance vector, respectively, e.g.,  $X = \begin{bmatrix}
x_{0}^\top x_{1}^\top \cdots x_{N}^\top
\end{bmatrix}^\top$,
while the matrices $\ScriptA\in\Re^{(N+1)n_x\times n_x}$, $\ScriptB \in\Re^{(N+1)n_x\times Nn_u}$, and $\ScriptD \in\Re^{(N+1)n_x\times Nn_x}$ are defined accordingly~\cite{okamoto2018Optimal}.
Note that 
\begin{subequations}\label{eq:Ex0x0x0WWW}
	\begin{align} 
	\mathbb{E}[x_0x_0^\top] &= \Sigma_0 + \mu_0\mu_0^\top,\\
	\mathbb{E}[x_0W^\top] &= 0, \\
	\mathbb{E}[WW^\top] &=\mathtt{blkdiag}(D_0D_0^\top,,\ldots,D_{N-1}D_{N-1}^\top). 
	\end{align}
\end{subequations}
We use the matrix $E_k = \left[0_{n_x,kn_x}, I_{n_x},0_{n_x,(N-k)n_x}\right]\in \Re^{n_x\times(N+1)n_x}$ and $F_k = \left[0_{n_u,kn_u}, I_{n_u},0_{n_u,(N-1-k)n_u}\right]\in \Re^{n_u\times Nn_u}$ such that $x_k = E_k X$ and $u_k = F_k U$.

\subsection{Input-Constrained OCS Problem \label{subsec:InputConstrainedOCS}}

In this section, we discuss the OCS with input hard constraints. 
We start by introducing a relaxed form of Problem~(\ref{prob:FiniteHorizonOptimalCovarianceSteeringProblem}) summarized in the following lemma.

\begin{Lemma}
	Given~(\ref{eq:Xdefinition}), (\ref{eq:Udefinition}), (\ref{eq:p_xj}), and (\ref{eq:xf}), Problem~(\ref{prob:FiniteHorizonOptimalCovarianceSteeringProblem}) can be converted to the following form with a convex relaxation of the terminal covariance constraint.
	\begin{subequations}\label{prob:OptimalCovarianceSteeringConvertedInputConstrained}
		\begin{align}
			&\minimize\; J(U) =\ \mathbb{E}\left[X^\top \bar{Q} X + U^\top \bar{R} U \right],\label{eq:ObjInputConstConvert1}\\
			&\textrm{subject to}\nonumber \\
			&\hspace{20pt}X = \ScriptA x_{0} + \ScriptB U+\ScriptD W,\\
			&\hspace{20pt}\Pr\left(\alpha_{x,j}^\top E_k X > \beta_{x,j} \right) \leq p_{x,j},   ~ j \in [0, N_s-1],  \label{eq:CCInputContConvert}\\
			&\hspace{20pt}\alpha_{u,s}^\top F_k U \leq \beta_{u,s},~~s \in [0, N_c-1], \label{eq:ICInputConstConvert}\\
			&\hspace{20pt}\mu_f = E_N\mathbb{E}[X],\label{eq:TermMeanInputConstConvert}\\
			&\hspace{20pt}\Sigma_f \succeq E_N\left(\mathbb{E}[XX^\top] - \mathbb{E}[X]\mathbb{E}[X]^\top\right) E_N^\top,\label{eq:TermCovInputConstConvert}
		\end{align}
	\end{subequations}
	where $\bar{Q} = \mathtt{blkdiag}(Q,\ldots,Q,0) \in \Re^{(N+1)n_x\times(N+1)n_x}$ and $\bar{R} = \mathtt{blkdiag}(R,\ldots,R) \in \Re^{Nn_u\times Nn_u}$.
\end{Lemma}
\begin{proof}
	The procedure to relax Problem~(\ref{prob:FiniteHorizonOptimalCovarianceSteeringProblem}) to (\ref{prob:OptimalCovarianceSteeringConvertedInputConstrained}) is straightforward using the result of~\cite{okamoto2018Optimalb} and by relaxing the original terminal covariance constraint $\Sigma_f = E_N\left(\mathbb{E}[XX^\top] - \mathbb{E}[X]\mathbb{E}[X]^\top\right) E_N^\top$ to~(\ref{eq:TermCovInputConstConvert}).
	We would like to highlight that the Gaussian distribution can be fully defined by its first two moments. 
	This relaxation makes the terminal state less distributed than the target distribution. 
	In many real-world applications~\cite{Ridderhof2019MinimumFuel} it may be preferable to achieve smaller uncertainty in the terminal state. 
\end{proof}

In order to solve Problem~(\ref{prob:OptimalCovarianceSteeringConvertedInputConstrained}), we propose a control law that is summarized in the following theorem, which is the main result of this paper.
This control policy is an extension of our previous controller~\cite{okamoto2018Optimalb} and is inspired from the approach in~\cite{paulson2017stochastic,hokayem2012stochastic}. 
\begin{Theorem}\label{theorem:InputConstrainedController}
	The control law
	\begin{align}\label{eq:inputConstrainedCSController}
		u_k = v_k + K_k z_k,
	\end{align}
	where $z_k$ is governed by the dynamics
	\begin{subequations}
		\begin{align} 
		z_{k+1} &= A z_k + \varphi(w_k), \label{eq:zDynamics}\\
		z_0 &= \varphi(\zeta_0), \quad \zeta_0 = x_0 - \mu_0, \label{eq:zInit}
		\end{align} 
	\end{subequations}
	where $\varphi(\cdot): \Re^d \mapsto \Re^d$ is an element-wise symmetric saturation function with pre-specified saturation value of the $i$th entry of the input $\zeta_i^{\rm max} > 0$ 
	as 
	\begin{align}\label{eq:saturationFunc}
		\varphi_i(\zeta) = \max(-\zeta_i^{\rm max}, \min(\zeta_i,\zeta_i^{\rm max}) ),
	\end{align}
	converts Problem~(\ref{prob:OptimalCovarianceSteeringConvertedInputConstrained}) to the following convex programming problem
	\begin{subequations}\label{prob:OptimalCovarianceSteeringFinalInputConstrained}
		\begin{align}
		&\minimize\;J(V, K, \Omega) =\nonumber \\
		& \trace\left(\bar{Q} \begin{bmatrix}I & \ScriptB K\end{bmatrix}\Sigma_{XX}	\begin{bmatrix}I \\ K^\top \ScriptB^\top \end{bmatrix}\right) + \trace\left(\bar{R}K\Sigma_{UU} K^\top\right) \nonumber \\ 
		& \hspace{20pt}+ (\ScriptA \mu_0 + \ScriptB V)^\top\bar{Q}(\ScriptA \mu_0 + \ScriptB V) +  V^\top\bar{R}V \label{eq:finalCostFunc}\\
		&\textrm{subject to }\nonumber \\
		&\alpha_{x,j}^\top E_k\left(\ScriptA \mu_0 + \ScriptB V\right)- \beta_{x,j} +\nonumber \\
		& \hspace{20pt}\sqrt{\frac{1-p_{x,j}}{p_{x,j}}} \|\Sigma_{XX}^{1/2}\begin{bmatrix}I & \ScriptB K\end{bmatrix}^\top
		E_k^\top\alpha_{x,j}\| \leq 0, \label{eq:OCSFIC_SC}\\
		&HF_kV + \Omega^\top\sigma \leq h,\\
		&HF_kK[
		\ScriptA\ \ScriptD
		] = \Omega^\top S,\\
		&\Omega \geq 0,\\
		& \mu_f = E_N\left(\ScriptA \mu_0 + \ScriptB V\right),\\
		&\Sigma_f \succeq E_N \begin{bmatrix}I & \ScriptB K\end{bmatrix}\Sigma_{XX}\begin{bmatrix}I \\ K^\top \ScriptB^\top \end{bmatrix} E_N^\top, \label{eq:OCSIC_TerminalCovConst}
		\end{align}
	\end{subequations}
	where $\Omega\in\Re^{2(N+1)n_x\times N_c}$ is a decision (slack) variable, 
	\begin{align}
		&\Sigma_{XX} = \nonumber\\
		&\begin{bmatrix}\ScriptA & \\ & \ScriptA\end{bmatrix}
		\begin{bmatrix}\Sigma_0 & \Expectation[\zeta_0\varphi(\zeta_0)^\top]\\ \Expectation[\varphi(\zeta_0)\zeta_0^\top] & \Expectation[\varphi(\zeta_0)\varphi(\zeta_0)^\top]\end{bmatrix}
		\begin{bmatrix}\ScriptA^\top & \\ & \ScriptA^\top\end{bmatrix}\nonumber \\
		& + \begin{bmatrix}\ScriptD & \\ & \ScriptD\end{bmatrix}
		\begin{bmatrix}\Expectation[WW^\top] & \Expectation[W\varphi(W)^\top]\\ 
		\Expectation[\varphi(W)W^\top] & \Expectation[\varphi(W)\varphi(W)^\top]\end{bmatrix}
		\begin{bmatrix}\ScriptD^\top & \\ & \ScriptD^\top\end{bmatrix},\label{eq:SigmaXX} 
	\end{align}
	\begin{align}
		&\Sigma_{UU} = 
		\ScriptA \Expectation[\varphi(\zeta_0)\varphi(\zeta_0)^\top]\ScriptA^\top + \ScriptD\Expectation[\varphi(W)\varphi(W)^\top]\ScriptD^\top.\label{eq:SigmaUU}
	\end{align}
	Furthermore, 
	\begin{subequations}\label{eq:Hh}
		\begin{align}
		H &= \begin{bmatrix}
		\alpha_{u,0},&
		\cdots,&
		\alpha_{u, N_c-1}
		\end{bmatrix}^\top \in \Re^{N_c\times n_u}, \\
		h &= \begin{bmatrix}
		\beta_{u,0},&
		\cdots &
		\beta_{u, N_c-1}
		\end{bmatrix}^\top \in \Re^{N_c},
		\end{align}
	\end{subequations}
	and
	\begin{align*}
		V &= \begin{bmatrix}
			v_0^\top &\cdots& v_{N-1}^\top
			\end{bmatrix}^\top,\\
		K &= \begin{bmatrix}
		\mathtt{blkdiag}(K_0,K_1,\ldots,K_{N-1}) & 0_{Nn_u\times n_x}
		\end{bmatrix}.
	\end{align*}
	
	In addition, $S \in \Re^{2(N+1)n_x\times(N+1)n_x}$ and $\sigma \in \Re^{2(N+1)n_x}$ are constant.
	Specifically, for $i = 1,\ldots, (N + 1)n_x$
	\begin{subequations}\label{eq:Ssigma}
		\begin{align}
		S_{2i-1} &= e_{2i-1}^\top,\	S_{2i} = -e_{2i}^\top,\\
		\sigma_{2i-1} &= \zeta_i^{\rm max},\ \sigma_{2i} = \zeta_i^{\rm max}.
		\end{align}
	\end{subequations}
	where $S_i$ denotes the $i$th row of $S$, and $e_i \in \Re^{2(N+1)n_x}$ is a unit vector with $i$th element 1.
\end{Theorem}

\begin{proof}
	It follows from~(\ref{eq:zDynamics}) and (\ref{eq:zInit}) that 
	\begin{align}
		Z = \ScriptA \varphi(\zeta_0) + \ScriptD \varphi(W),
	\end{align}
	where $ Z = \begin{bmatrix} z_0^\top & \ldots & z_N^\top \end{bmatrix}^\top \in \Re^{(N+1)n_x}$.
	In addition, $U$ is represented as $U = V + K Z$,
	and thus,
	\begin{align} \label{eq:newControlPolicy}
		U = V + K(\ScriptA \varphi(\zeta_0) + \ScriptD \varphi(W)).
	\end{align}
	Since the saturation function~(\ref{eq:saturationFunc}) is an odd function and $\zeta_0$ is zero-mean Gaussian distributed,
	\begin{align*}
		\Expectation[\varphi(\zeta_0)] = 0,\qquad
		\Expectation[\varphi(W)] = 0, 
	\end{align*}
	and hence,
	\begin{align*}
		\Expectation[U] = V,\quad
		\tilde{U} = U - \Expectation[U]
		=K(\ScriptA \varphi(\zeta_0) + \ScriptD \varphi(W)).
	\end{align*}
	Thus, it follows from~(\ref{eq:X=Ax0+BU+DW}) that 
	\begin{align}
		\bar{X} &= \Expectation[X] = \ScriptA \mu_0 + \ScriptB V,\label{eq:barXVK}\\
		\tilde{X} &= X -\Expectation[X] \nonumber \\
		&\hspace{-15pt}= \begin{bmatrix}I & \ScriptB K\end{bmatrix} \left(
		\begin{bmatrix}\ScriptA & \\ & \ScriptA\end{bmatrix}
		\begin{bmatrix}\zeta_0 \\ \varphi(\zeta_0)\end{bmatrix}
		+ 
		\begin{bmatrix}\ScriptD & \\ & \ScriptD\end{bmatrix}
		\begin{bmatrix}W \\ \varphi(W)\end{bmatrix}          
		\right).\label{eq:tildeXVK}  
	\end{align}
	It also follows from~(\ref{eq:Ex0x0x0WWW}),~(\ref{eq:zInit}),~(\ref{eq:SigmaXX}), and~(\ref{eq:SigmaUU}) that 
	\begin{align}
		\Expectation[\tilde{X}\tilde{X}^\top] &= \begin{bmatrix}I & \ScriptB K\end{bmatrix}	\Sigma_{XX}	\begin{bmatrix}I \\ K^\top \ScriptB^\top\end{bmatrix},\label{eq:EXX}\\
		\Expectation[\tilde{U}\tilde{U}^\top] &=  K\Sigma_{UU} K^\top.\label{eq:EUU}
	\end{align}
	Following~\cite{okamoto2018Optimal}, it can be shown that the cost function~(\ref{eq:ObjInputConstConvert1}) may be written as 
	\begin{align} \label{eq:objInputConstConvert2}
		J(V,\tilde{U}) &= \trace(\bar{Q}\Expectation[\tilde{X}\tilde{X}^\top]) + \bar{X}^\top \bar{Q} \bar{X} \nonumber \\
		&\hspace{20pt}+ \trace(\bar{R}\Expectation[\tilde{U}\tilde{U}^\top]) + V^\top \bar{R} V.
	\end{align}
	Using~(\ref{eq:barXVK}), (\ref{eq:EXX}), and (\ref{eq:EUU}), we can further convert~(\ref{eq:objInputConstConvert2}) to~(\ref{eq:finalCostFunc}).
	
	Next, we discuss the conversion of the chance constraint~(\ref{eq:CCInputContConvert}) to a more amenable form to facilitate the numerical solution of the optimization problem (\ref{prob:OptimalCovarianceSteeringFinalInputConstrained}).
	First, notice that $\alpha_j^\top E_k X$ is a univariate random variable with mean $\alpha_{x,j}^\top E_k \Expectation[X]$ and variance $\alpha_{x,j}^\top E_k \Expectation[\tilde{X}\tilde{X}^\top] E_k^\top \alpha_{x,j}$.
	It follows from the Chebyshev-Cantelli inequality~(Lemma~\ref{Lemma:Cantelli} in the Appendix) that 
	\begin{align*}
	&\Pr\bigg(\alpha_{x,j}^\top E_k X \leq \alpha_{x,j}^\top E_k \Expectation[X] \nonumber\\ &\hspace{20pt}+\sqrt{\frac{1-p_{x,j}}{p_{x,j}}}\sqrt{ \alpha_{x,j}^\top E_k \Expectation[\tilde{X}\tilde{X}^\top]E_k^\top\alpha_{x,j}^\top}\bigg) > 1-p_{x,j}.
	\end{align*}
	Therefore, the inequality~(\ref{eq:CCInputContConvert}) is satisfied if
	\begin{align*}
	\alpha_{x,j}^\top E_k \Expectation[X] +\sqrt{\frac{1-p_{x,j}}{p_{x,j}}}\sqrt{ \alpha_{x,j}^\top E_k \Expectation[\tilde{X}\tilde{X}^\top]E_k^\top\alpha_{x,j}^\top} \leq \beta_{x,j},
	\end{align*}
	is satisfied, which is, from (\ref{eq:barXVK}), (\ref{eq:tildeXVK}), and a similar discussion to~\cite{okamoto2018Optimal}, equivalent to a second order cone constraint in terms of $V$ and $K$~(\ref{eq:OCSFIC_SC}).

	In addition, using~(\ref{eq:barXVK}) and (\ref{eq:EXX}), the terminal state constraints~(\ref{eq:TermMeanInputConstConvert}) and (\ref{eq:TermCovInputConstConvert}) can be written as
	\begin{subequations}
		\begin{align}
		\mu_f &= E_N\left(\ScriptA \mu_0 + \ScriptB V\right),\\
		\Sigma_f &\succeq E_N \begin{bmatrix}I & \ScriptB K\end{bmatrix}
		\Sigma_{XX}
		\begin{bmatrix}I \\ K^\top \ScriptB^\top \end{bmatrix} E_N^\top.\label{eq:XNNew}
		\end{align}
	\end{subequations}
	Using the Schur complement, (\ref{eq:XNNew}) can be further converted to 
	\begin{align*}
		\begin{bmatrix}
		\Sigma_f & E_N \begin{bmatrix}I & \ScriptB K\end{bmatrix}\Sigma_{XX}^{1/2}\\
		\Sigma_{XX}^{1/2}	\begin{bmatrix}I & \ScriptB K\end{bmatrix}^\top E_N^\top & I 
		\end{bmatrix} \succeq 0.
	\end{align*}
	Note that $\Sigma_{XX}\succeq 0$ (see Lemma~\ref{lemma:SigmaXXisPositiveSemiDef} in the Appendix).

	Finally, we rewrite the input hard constraint~(\ref{eq:ICInputConstConvert}) as follows.
	Using~(\ref{eq:Hh}), we first rewrite~(\ref{eq:ICInputConstConvert}) to the following equivalent form
	\begin{equation}\label{eq:InputConstraintsHu<h}
		H F_k U \leq h.
	\end{equation}
	Then, using~(\ref{eq:newControlPolicy}), this inequality is further converted to
	\begin{equation}
	HF_k\left(V + K \left(\mathcal{A}\varphi(\zeta_0) + \mathcal{D}\varphi(W)\right) \right) \leq h,
	\end{equation}
	and thus, 
	\begin{equation} \label{eq:InputHardConstRaw}
	HF_kV + HF_kK \begin{bmatrix}	\mathcal{A} & \mathcal{D} \end{bmatrix}\begin{bmatrix}
	\varphi(\zeta_0) \\\varphi(W) \end{bmatrix}  \leq h.
	\end{equation}
	In addition, using~(\ref{eq:Ssigma}), the condition~($\ref{eq:saturationFunc}$) can be represented as 
	\begin{equation}
	S \begin{bmatrix} \varphi(\zeta_0) \\\varphi(W) \end{bmatrix} \leq \sigma.
	\end{equation}
	It follows from the discussion in~\cite{paulson2017stochastic} that the constraint~(\ref{eq:InputHardConstRaw}) can be converted to 
	\begin{subequations}
		\begin{align}
		&HF_kV + \Omega^\top\sigma \leq h,\\
		&HF_kK[
		\ScriptA\ \ScriptD
		] = \Omega^\top S,\\
		&\Omega \geq 0.
		\end{align}
	\end{subequations}
	In summary, we have converted Problem~(\ref{prob:OptimalCovarianceSteeringConvertedInputConstrained}) to Problem~(\ref{prob:OptimalCovarianceSteeringFinalInputConstrained}).
\end{proof}

Note that the values of $\Expectation[\zeta_0\varphi(\zeta_0)^\top]$, $\Expectation[\varphi(\zeta_0)\varphi(\zeta_0)^\top]$, $\Expectation[W\varphi(W)^\top]$, and $\Expectation[\varphi(W)\varphi(W)^\top]$ can be obtained using Monte Carlo or Lemma~\ref{lemma:Integrals} in the Appendix.

\begin{Remark}
	When designing $u_k$ in~(\ref{eq:inputConstrainedCSController}), the value of $z_k$ can be obtained from $z_{k-1}$ and $\varphi(w_{k-1})$.
	The noise $w_{k-1}$ is obtained from the system dynamics~(\ref{eq:SystemDynamics}), i.e., 
	\begin{align}
	w_{k-1} = x_{k} - A_{k-1}x_{k-1} - B_{k-1}u_{k-1},
	\end{align}
	and thus, $\varphi(w_{k-1})$ can be computed before computing the value of $u_k$ at time step $k$. 
\end{Remark}

\begin{Remark}
	Instead of~(\ref{eq:originalCostFunc}) or (\ref{eq:finalCostFunc}), one can separately design the mean and covariance steering cost, i.e.,  	
	\begin{align*}
		&J(V,K) = \trace\left(\bar{Q}_v \begin{bmatrix}I & \ScriptB K\end{bmatrix}
		\Sigma_{XX}
		\begin{bmatrix}I & \ScriptB K\end{bmatrix}^\top\right) + V^\top\bar{R}_mV \nonumber \\
		&+(\ScriptA \mu_0 + \ScriptB V)^\top\bar{Q}_m(\ScriptA \mu_0 + \ScriptB V) + \trace\left(\bar{R}_vK\Sigma_{UU} K^\top\right),
	\end{align*}
	where $\bar{Q}_m$ $\in$ $\Re^{(N+1)n_x\times(N+1)n_x}$, $\bar{Q}_v$ $\in$ $\Re^{(N+1)n_x\times(N+1)n_x}$, $\bar{R}_m$ $\in$ $\Re^{Nn_u\times Nn_u}$, and $\bar{R}_v$ $\in$ $\Re^{Nn_u\times Nn_u}$ are block diagonal matrices, e.g., $\bar{Q}_m = \mathtt{blkdiag}(Q_m$,$\ldots$, $Q_m, 0)$. 
\end{Remark}

\begin{Remark}
	Because the terminal covariance constraint~(\ref{eq:OCSIC_TerminalCovConst}) can be formulated as a linear matrix inequality (LMI) constraint, in order to solve Problem (\ref{prob:OptimalCovarianceSteeringFinalInputConstrained}), we need to use an SDP solver such as MOSEK~\cite{mosek}.
\end{Remark}

\section{Numerical Simulation\label{sec:Numerical Simulation}}
In this section we validate the proposed OCS algorithm using a numerical example. 
We consider a path-planning problem for a vehicle under the following double integrator dynamics with $x_k = [x,y,v_x,v_y]^\top\in \Re^{4}$, $u_k =[a_x,a_y]^\top\in \Re^{2}$, $w_k \in \Re^{4}$ and
\begin{align*}
	&A = \begin{bmatrix}
	1 & 0 & \Delta t & 0\\
	0 & 1 & 0 & \Delta t\\
	0 & 0 & 1 & 0\\
	0 & 0 & 0 & 1
	\end{bmatrix},\;
	B = \begin{bmatrix}
	\Delta t^2/2 & 0\\
	0 &\Delta t^2/2\\
	\Delta t & 0\\
	0 & \Delta t
	\end{bmatrix},
\end{align*}
and $D = \mathtt{blkdiag}(0.01, 0.01, 0.01, 0.01)$,
where $\Delta t$ is the time-step size set to $\Delta t = 0.2$. 
The feasible state space is 
\begin{align*}
	0.2(x-1) \leq y \leq -0.2(x-1).
\end{align*}
The mean and the covariance of the initial and target terminal distributions are set to
\begin{align*}
	&\mu_0 = [-10, 1, 0, 0],\;\Sigma_0 = \mathtt{blkdiag}(0.05, 0.05, 0.01, 0.01),\\
	&\mu_f = [0, 0, 0, 0], \; \Sigma_f = \mathtt{blkdiag}(0.025, 0.025, 0.005, 0.005).
\end{align*}
The cost function weights are chosen as 
\begin{align*}
	Q = \mathtt{blkdiag}(0.5,4.0,0.05,0.05),\; R = \mathtt{blkdiag}(20,20).
\end{align*}
The horizon is set to $N=20$, and the probability threshold to $p_{x,j} = 0.05$ for $j = 0, 1$.
Finally, we restrict the maximum acceleration at $U_{\rm max}= $ 2.9 $\mathrm{m/s^2}$ along each axis, i.e.,
\begin{align}\label{eq:controlConstraintEq}
	\| u_k \|_\infty \leq U_{\rm max},
\end{align}
for $k= 0,\ldots,N-1$.
We also set the saturation function~(\ref{eq:saturationFunc}) to be saturated when the input exceeds the 3$\sigma$ values.
We employ YALMIP~\cite{yalmip} along with MOSEK~\cite{mosek} to solve this problem.

We first show the results when the controller from~\cite{okamoto2018Optimalb} is used in Fig.~\ref{fig:ResultsFree}.
The trajectories are depicted in Fig.~\subref*{fig:Traj_F}.
Red circles denote the initial and the target distributions, and blue ellipses represent the 3$\sigma$ confidence region at each time step. 
We also illustrate randomly picked 100 trajectories with gray lines and observe that the state chance constraints are satisfied. 
Note, however, that the controller in~\cite{okamoto2018Optimalb} cannot deal with input hard constraints.
Figure~\subref*{fig:ACC_F} depicts the acceleration commands of the 100 sample trajectories with gray lines along with the acceleration limits $\pm$ 2.9 $\mathrm{m/s^2}$ with red dashed lines. 
The cost for this scenario is 2,285.

Next, we show the results when the newly developed OCS controller in Theorem~\ref{theorem:InputConstrainedController} is used. 
Figure~\ref{fig:ResultsConstrained} depicts the results. 
While, as shown in Fig.~\subref*{fig:Traj_C}, the state chance constraints are satisfied, the control commands depicted in Fig.~\subref*{fig:ACC_C} satisfy the input hard constraint~(\ref{eq:controlConstraintEq}).
The cost for this constrained scenario is 2,301, which is, as expected, slightly larger than the cost for the unconstrained case owing to the additional input constraint~(\ref{eq:controlConstraintEq}).

\begin{figure}
	\centering
	\subfloat[Trajectories.\label{fig:Traj_F}]{\centering\includegraphics[width=0.9\columnwidth]{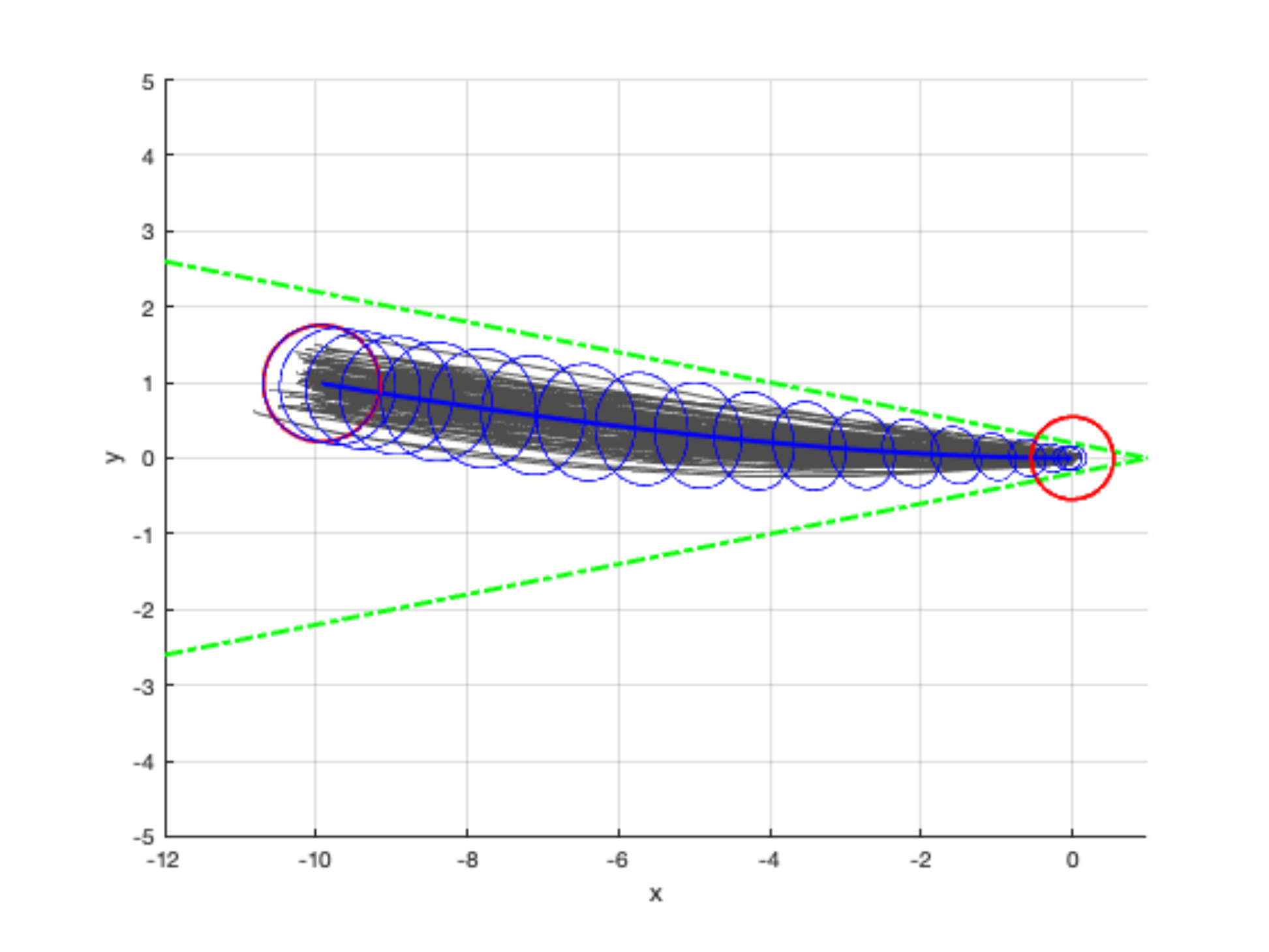}}\\
	\subfloat[Input acceleration commands.\label{fig:ACC_F}]{\centering\includegraphics[width=0.9\columnwidth]{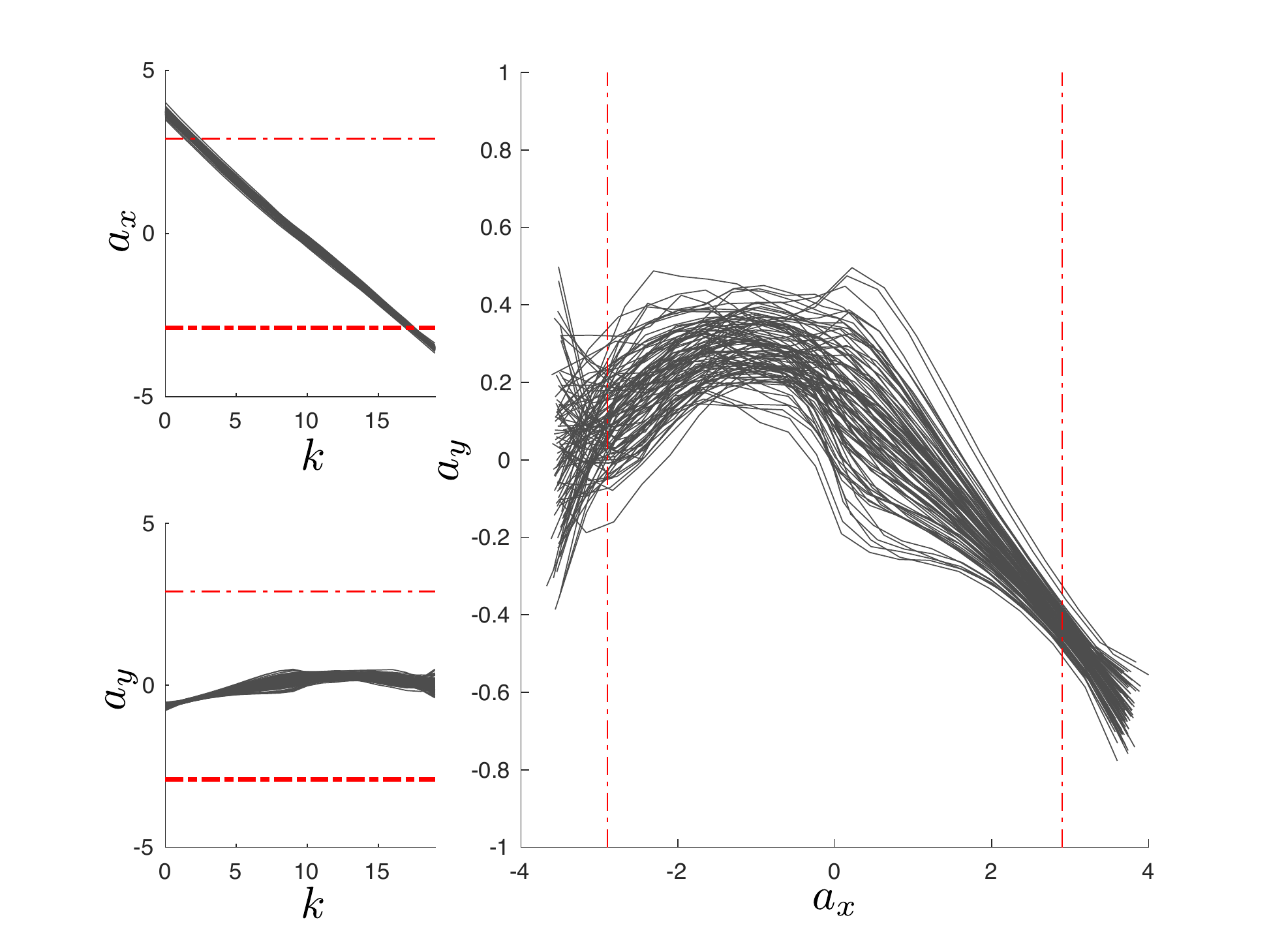}}
	\caption{Results when input is constraint free~\cite{okamoto2018Optimalb}.\label{fig:ResultsFree}}
\end{figure}

\begin{figure}
	\centering
	\subfloat[Trajectories.\label{fig:Traj_C}]{\centering\includegraphics[width=0.9\columnwidth]{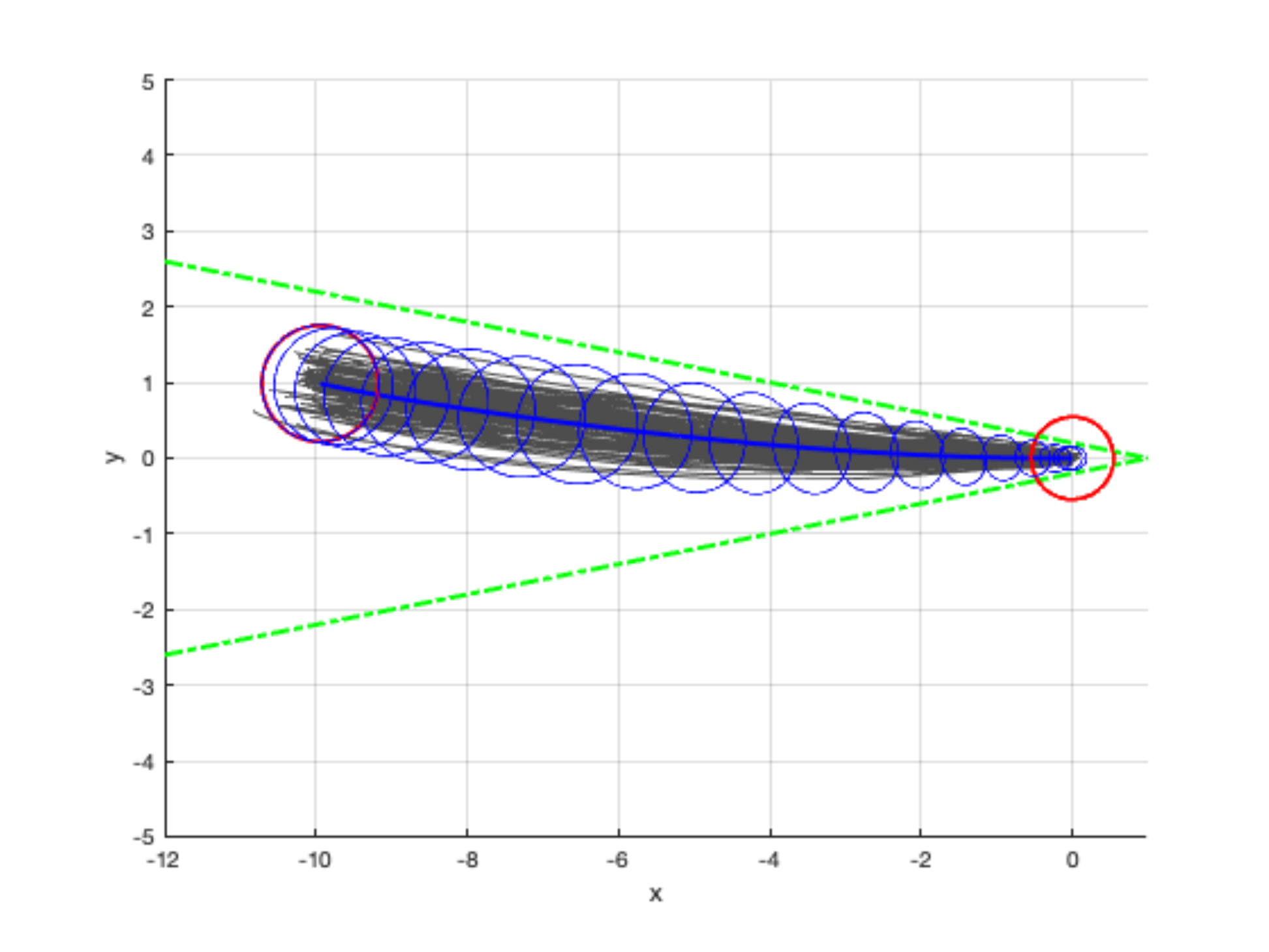}}\\
	\subfloat[Input acceleration commands.\label{fig:ACC_C}]{\centering\includegraphics[width=0.9\columnwidth]{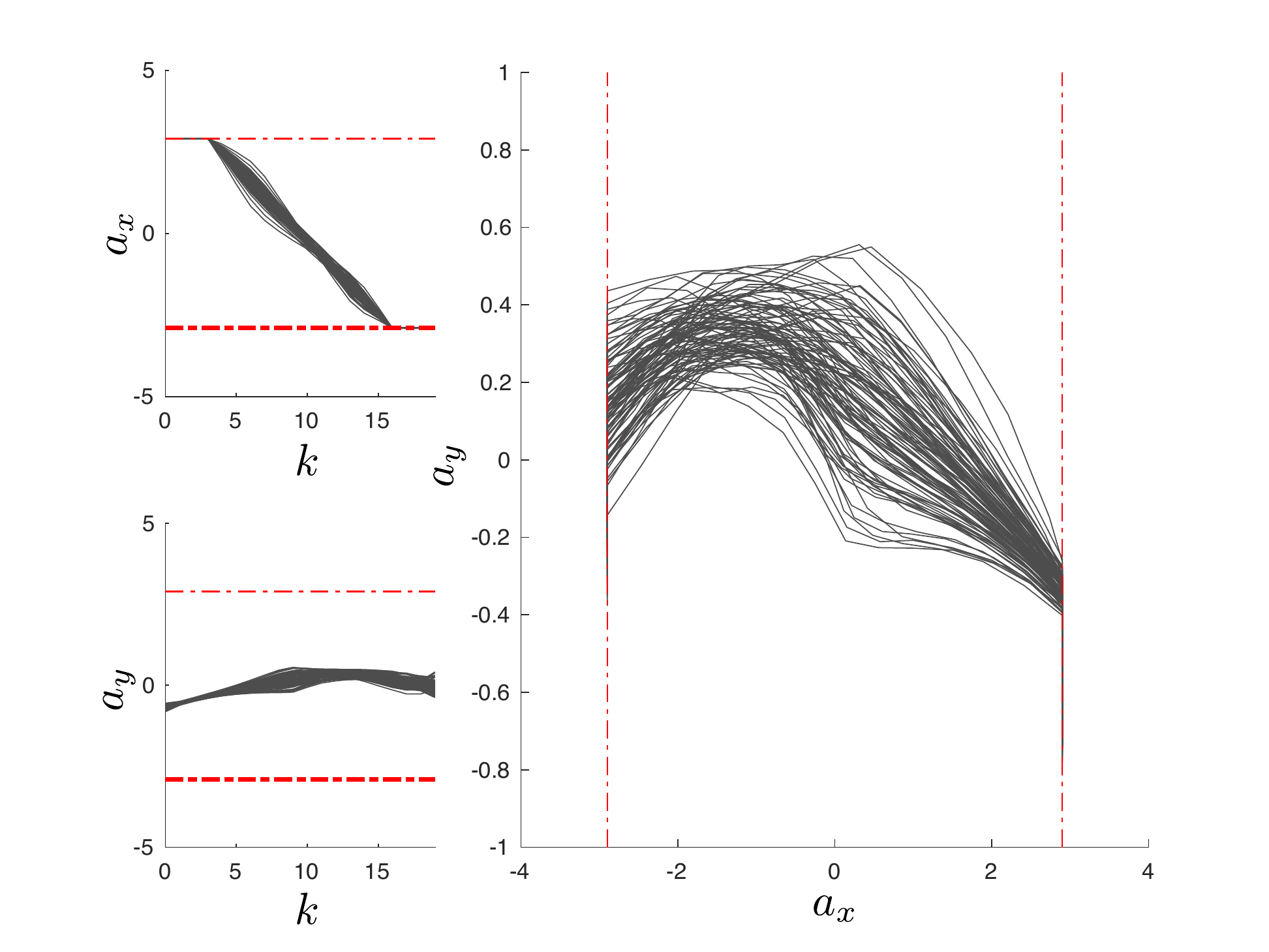}}
	\caption{Results when input constraints are imposed.\label{fig:ResultsConstrained}}
\end{figure}

\section{Summary}\label{sec:Summary}

In this work, we have addressed the problem of OCS under state chance constraints and input hard constraints.
Similarly to our previous works~\cite{okamoto2018Optimal,okamoto2018Optimalb}, we solved this problem by converting the original problem into a convex programing problem.
The input hard constraints are formulated using saturation functions to limit the effect of possibly unbounded disturbance.
Numerical simulations show that the proposed algorithm successfully constructs control commands that satisfy the state and input constraints.
Future work includes OCS with measurement noise. 

\vspace*{-1ex}
\bibliographystyle{IEEEtran}
\bibliography{InputConstrainedCS}

\section*{APPENDIX}

\begin{Lemma}\label{lemma:Integrals}
	If a random variable $z\in\Re$ is sampled from a Gaussian distribution $\mathcal{N}(0, \sigma_z^2)$ and the saturation function $\varphi(\cdot)$ is such that 
	\begin{align*}
		\varphi(z) = \max(-\zeta, \min(z, \zeta)),
	\end{align*}
	where $\zeta > 0$ is a constant, then 
	\begin{align*}
	&\Expectation[\varphi(z)^2] =\zeta^2 + (\sigma_z^2 - \zeta^2)\erf\left(\frac{\zeta}{\sqrt{2}\sigma_z}\right) - \frac{2\zeta\sigma_z}{\sqrt{2\pi}}e^{-\frac{\zeta^2}{2\sigma_z^2}}.\\
	&\Expectation[z\varphi(z)] = \sigma_z^2\erf\left(\frac{\zeta}{\sqrt{2}\sigma_z}\right).\label{eq:Ephizz}\\
	\end{align*}
\end{Lemma}
\begin{proof}
	The proof is straightforward from standard integral calculus. 
\end{proof}

\begin{Lemma}\label{lemma:SigmaXXisPositiveSemiDef}
	The matrix $\Sigma_{XX}$ defined as
	\begin{align*}
	&\Sigma_{XX} =\nonumber \\
	&\begin{bmatrix}\ScriptA & \\ & \ScriptA\end{bmatrix}
	\begin{bmatrix}\Expectation[\zeta_0\zeta_0^\top] & \Expectation[\zeta_0\varphi(\zeta_0)^\top]\\ \Expectation[\varphi(\zeta_0)\zeta_0^\top] & \Expectation[\varphi(\zeta_0)\varphi(\zeta_0)^\top]\end{bmatrix}
	\begin{bmatrix}\ScriptA^\top & \\ & \ScriptA^\top\end{bmatrix}\nonumber \\
	&+ \begin{bmatrix}\ScriptD & \\ & \ScriptD\end{bmatrix}
	\begin{bmatrix}\Expectation[WW^\top] & \Expectation[W\varphi(W)^\top]\\ \Expectation[\varphi(W)W^\top] & \Expectation[\varphi(W)\varphi(W)^\top]\end{bmatrix}
	\begin{bmatrix}\ScriptD^\top & \\ & \ScriptD^\top\end{bmatrix}.
	\end{align*}
	is symmetric positive definite. 
\end{Lemma}
\begin{proof}
	Consider vectors $x$ $\in$ $\Re^{(N+1)n_x}$ and  $y$ $\in$ $\Re^{(N+1)n_x}$.
	Then, 
	\begin{align*}
	&\begin{bmatrix} x^\top &y^\top\end{bmatrix} \Sigma_{XX} \begin{bmatrix}x\\ y\end{bmatrix}\nonumber \\
	&= \begin{bmatrix} x^\top \ScriptA &y^\top\ScriptA\end{bmatrix}
	\begin{bmatrix}\Expectation[\zeta_0\zeta_0^\top]  & \Expectation[\zeta_0\varphi(\zeta_0)^\top]\\ \Expectation[\varphi(\zeta_0)\zeta_0^\top] & \Expectation[\varphi(\zeta_0)\varphi(\zeta_0)^\top]\end{bmatrix}
	\begin{bmatrix}\ScriptA^\top x\\ \ScriptA^\top y\end{bmatrix} \nonumber \\
	& +\begin{bmatrix} x^\top \ScriptD &y^\top\ScriptD\end{bmatrix}
	\begin{bmatrix}\Expectation[WW^\top] & \Expectation[W\varphi(W)^\top]\\ \Expectation[\varphi(W)W^\top] & \Expectation[\varphi(W)\varphi(W)^\top]\end{bmatrix}
	\begin{bmatrix}\ScriptD^\top x\\ \ScriptD^\top y\end{bmatrix},\\
	&=\begin{bmatrix} x^\top \ScriptA &y^\top\ScriptA\end{bmatrix}
	\Expectation\left[\begin{bmatrix} \zeta_0\\\varphi(\zeta_0)\end{bmatrix}\begin{bmatrix} \zeta_0^\top&\varphi(\zeta_0)^\top\end{bmatrix}\right]
	\begin{bmatrix}\ScriptA^\top x\\ \ScriptA^\top y\end{bmatrix} \nonumber \\
	& + \begin{bmatrix} x^\top \ScriptD &y^\top\ScriptD\end{bmatrix}
	\Expectation\left[\begin{bmatrix} W\\\varphi(W)\end{bmatrix}\begin{bmatrix} W^\top&\varphi(W)^\top\end{bmatrix}\right]
	\begin{bmatrix}\ScriptD^\top x\\ \ScriptD^\top y\end{bmatrix}, 
	\end{align*}

	\begin{align*}
	&=\Expectation\left[(x^\top\ScriptA \zeta_0 + y^\top\ScriptA\varphi(\zeta_0))
	(\zeta_0^\top\ScriptA^\top x + \varphi(\zeta_0)^\top\ScriptA^\top y)\right] \nonumber \\
	& + \Expectation\left[(x^\top\ScriptD W + y^\top\ScriptD\varphi(W))
	(W^\top\ScriptD^\top x + \varphi(W)^\top\ScriptD^\top y)\right],\\
	&=\Expectation\left[(x^\top\ScriptA \zeta_0 + y^\top\ScriptA\varphi(\zeta_0))^2\right] \nonumber \\
	&\hspace{20pt}+ \Expectation\left[(x^\top\ScriptD W + y^\top\ScriptD\varphi(W))^2\right] \geq 0.
	\end{align*}

\end{proof}
\begin{Lemma}[Chebyshev-Cantelli inequality~\cite{marshall1960multivariate}\label{Lemma:Cantelli}]
	Let $z \in \Re$ be a random variable with mean $\mu_z$ and variance $\Sigma_z$. Then, for any $c\geq 0$, the following inequality holds
	\begin{align}
		\Pr(z \geq \mu_z + c) \leq \frac{\Sigma_z}{\Sigma_z + c^2}.
	\end{align} 
\end{Lemma}
	
The following lemma describes the derivation of the chance constraints~(\ref{eq:OCSFIC_SC}) in detail. 

\begin{Lemma}
The chance constraint
\begin{equation} \label{eq:new1}
\Pr\left(\alpha_{x,j}^\top E_k X > \beta_{x,j} \right) \leq p_{x,j}, 
\end{equation}
can be satisfied if 
\begin{align*}
	&\alpha_{x,j}^\top E_k\left(\ScriptA \mu_0 + \ScriptB V\right)- \beta_{x,j} +\nonumber \\
		& \hspace{20pt}\sqrt{\frac{1-p_{x,j}}{p_{x,j}}} \|\Sigma_{XX}^{1/2}\begin{bmatrix}I & \ScriptB K\end{bmatrix}^\top
		E_k^\top\alpha_{x,j}\| \leq 0
\end{align*}
\end{Lemma}

\begin{proof}
	It follows from Lemma~\ref{Lemma:Cantelli} that the following inequality on the random scalar variable $\alpha_j^\top E_k X$ holds
	\begin{align}
		&\Pr(\alpha_j^\top E_k X \leq \alpha_j^\top E_k \Expectation[X] + c) > \nonumber \\
		&\hspace{50pt} 1 - \frac{\alpha_j^\top E_k \Expectation[\tilde{X}\tilde{X}^\top]E_k^\top \alpha_j^\top }{\alpha_j^\top E_k \Expectation[\tilde{X}\tilde{X}^\top]E_k^\top \alpha_j^\top + c^2},
	\end{align}
	In addition, inequality (\ref{eq:new1}) can be equivalently rewritten as
	\begin{align}\label{ineq:Apdx1}
		\Pr(\alpha_{x,j}^\top E_k X \leq \beta_{x,j}) > 1 - p_{x,j}.
	\end{align}
	We wish to compute $c>0$ that satisfies
	\begin{align}
		p_{x,j} = \frac{\alpha_j^\top E_k \Expectation[\tilde{X}\tilde{X}^\top]E_k^\top \alpha_j^\top }{\alpha_j^\top E_k \Expectation[\tilde{X}\tilde{X}^\top]E_k^\top \alpha_j^\top + c^2},
	\end{align}
	and obtain 
	\begin{align}
		c = \sqrt{\frac{1-p_{x,j}}{p_{x,j}}}\sqrt{ \alpha_{x,j}^\top E_k \Expectation[\tilde{X}\tilde{X}^\top]E_k^\top\alpha_{x,j}^\top}.
	\end{align}
	Therefore, the following inequality holds
	\begin{align}\label{ineq:Apdx2}
	&\Pr\bigg(\alpha_{x,j}^\top E_k X \leq \alpha_{x,j}^\top E_k \Expectation[X] \nonumber\\ &\hspace{20pt}+\sqrt{\frac{1-p_{x,j}}{p_{x,j}}}\sqrt{ \alpha_{x,j}^\top E_k \Expectation[\tilde{X}\tilde{X}^\top]E_k^\top\alpha_{x,j}^\top}\bigg) > 1-p_{x,j}.
	\end{align}
	By comparing~(\ref{ineq:Apdx1}) and~(\ref{ineq:Apdx2}), it follows that (\ref{ineq:Apdx1}) is satisfied if
	\begin{align}
		\alpha_{x,j}^\top E_k \Expectation[X] +\sqrt{\frac{1-p_{x,j}}{p_{x,j}}}\sqrt{ \alpha_{x,j}^\top E_k \Expectation[\tilde{X}\tilde{X}^\top]E_k^\top\alpha_{x,j}^\top} \leq \beta_{x,j},
	\end{align}
	which is equivalent to the second order cone constraint in terms of $V$ and $K$~(\ref{eq:OCSFIC_SC}).
\end{proof}

\end{document}